\documentclass[12pt]{amsart}
\usepackage{amssymb,amsfonts,amsmath,amsthm,cite,verbatim}
\usepackage{xcolor}
\usepackage{graphicx}
\usepackage{array}
\usepackage{booktabs}
\usepackage[linesnumbered,ruled,vlined]{algorithm2e}
\usepackage[left=2.5cm, right=2.5cm, top=3.5cm, bottom=3.5cm]{geometry}

\usepackage{hyperref}

\usepackage[normalem]{ulem}

\newtheorem{theorem}{Theorem}[section]
\newtheorem{theo}[theorem]{Theorem}
\newtheorem{prop}[theorem]{Proposition}
\newtheorem{lem}[theorem]{Lemma}
\newtheorem{rem}[theorem]{Remark}

\newtheorem{cor}[theorem]{Corollary}
\newtheorem{exam}[theorem]{Example}

\makeatletter \@addtoreset{equation}{section}

\SetCommentSty{mycommfont}

\SetKwInput{KwInput}{Input}                
\SetKwInput{KwOutput}{Output}              

\DeclareMathOperator*{\CT}{CT}
\DeclareMathOperator*{\res}{res}

\newcommand{\NN}{\mathbb{N}}
\newcommand{\ZZ}{\mathbb{Z}}
\newcommand{\QQ}{\mathbb{Q}}
\newcommand{\RR}{\mathbb{R}}

\def\lcm{\mathrm{lcm}}

\def\y{\boldsymbol{y}}
\def\a{\boldsymbol{a}}
\def\m{\boldsymbol{m}}
\def\e{\mathrm{e}}
\def\d{\mathrm{d}}
\def\ind{\mathrm{ind}}
\def\sgn{\mathrm{sgn}}
\def\Uni{\mathrm{Uni}}
\def\NUni{\mathrm{NonUni}}
\def\f{\mathfrak{f}}
\def\P{\mathcal{P}}

\def\F{\mathcal{F}}
\def\B{\mathcal{B}}
\def\S{\mathcal{S}}
\def\K{\mathcal{K}}

\def\D{\mathcal{D}}

\title{An algebraic combinatorial approach to Sylvester's denumerant}

\author{Guoce Xin$^{1}$ and Chen Zhang$^{2,*}$}

\address{ $^{1,2}$School of Mathematical Sciences, Capital Normal University, Beijing 100048, PR China}

\email{$^1$\texttt{guoce\_xin@163.com}\ \& $^2$\texttt{ch\_enz@163.com}}

\date{December 4, 2023}

\thanks{$*$ Corresponding author.}

\begin{document}

\begin{abstract}
For a positive integer sequence $\a=(a_1, \dots, a_{N+1})$, Sylvester's denumerant $E(\a; t)$ counts the number of nonnegative integer solutions to $\sum_{i=1}^{N+1} a_i x_i = t$ for a nonnegative integer $t$. It has been extensively studied and a well-known result asserts that $E(\a; t)$ is a quasi-polynomial in $t$ of degree $N$. A milestone is Baldoni et al.'s polynomial algorithm in 2015 for computing the top $k$ coefficients when $k$ is fixed. Their development uses heavily lattice point counting theory in computational geometry. In this paper, we explain their work in the context of algebraic combinatorics and simplify their computation. Our work is based on constant term method, Barvinok's unimodular cone decomposition, and recent results on fast computation of generalized Todd polynomials. We develop the algorithm \texttt{CT-Knapsack}, together with an implementation in \texttt{Maple}. Our algorithm avoids plenty of repeated computations and is hence faster.
\end{abstract}

\maketitle

\noindent
\begin{small}
 \emph{Mathematic subject classification}: Primary 05--08; Secondary 05--04, 05A17.
\end{small}

\noindent
\begin{small}
\end{small}

\section{Introduction}
For a positive integer sequence $\a  = (a_1, \dots, a_{N+1})$ and a nonnegative integer $t$,
the Sylvester's \emph{denumerant} $E(\a;t)$ denotes the number of nonnegative integer solutions to the equation $\sum_{i=1}^{N+1} a_i x_i = t$.
It is well-known that $E(\a;t)$ is a \emph{quasi-polynomial} in $t$ of degree $N$, i.e., $E(\a;t) = \sum_{m=0}^{N} E_m(\a; t) t^m$,
where $E_m(\a; t)$ is a periodic function in $t$ with period $\lcm(\a)$ for each $m$.
This period was first given by Sylvester and Cayley and then proved by Bell in a simpler way in 1943 (See \cite{1943-Bell} and references therein).
Bell also remarked that $\lcm(\a)$ is only the worst case period, and the actual period is usually smaller.

In what follows, we focus on the case when $\gcd(\a)=1$.
The general case when $\gcd(\a)=d$ follows by the natural formula $E(\a; dt) = E(\a /d; t)$.
Denote by
\begin{equation}\label{e-Fofa}
F(\a; z,t) := \frac{z^{-t-1}}{\prod_{i=1}^{N+1} (1-z^{a_i})}=\sum_{x_1,\dots, x_{N+1} \ge 0} z^{\sum_{i=1}^{N+1} a_i x_i - t-1}.
\end{equation}
Then we clearly have
\[
E(\a;t)= \res_{z=0} F(\a; z,t) = \CT_z z F(\a; z,t),
\]
where $\CT_z f(z) $ denotes the constant term of the Laurent series expansion (at $z=0$) of $f(z)$, and
$\res_{z=z_0}f(z)$ denotes the residue of $f(z)$ when expanded as a Laurent series at $z=z_0$. More precisely,
$\res_{z=z_0} \sum_{i \ge i_0} c_i (z-z_0)^i = c_{-1}$.

Sylvester's denumerant is associated with many problems.

i) The Frobenius problem (or coin-exchange problem) is to find the largest $t$ such that $E(\a; t)=0$. A comprehensive overview and some applications of the Frobenius problem can be found in \cite{book-Frobenius}.

ii) The gaps (or holes) of a numerical semigroup $\langle \a \rangle = \NN a_1 + \cdots + \NN a_{N+1}$ refer to those $t$ such that $E(\a; t)=0$.
For general references of numerical semigroups, see \cite{book-NumericalSemigroups,book-Frobenius}.

iii) Thinking about the denumerant problem in another way, if $a_i$'s are pairwise relatively prime, then the Fourier-Dedekind sums form building blocks of $E(\a; t)$. See \cite{book-Beck} for details and further references.

iv) In integer optimization, deciding if $E(\a; t) > 0$ is an important integer feasibility problem. See, e.g., \cite{2004-HardKnapsack}.

It is well-known that both the problems of deciding whether $E(\a; t) > 0$ for a given $t$ and finding the Frobenius number of $\a$ are NP-hard,
and computing $E(\a; t)$ is $\#$P-hard.
The denumerant problem has been extensively studied, not only in computational geometry but also in algebraic combinatorics.
See, e.g., \cite{2002-Agnarsson,2018-Aguilo,2015-Baldoni-Denumerant,1943-Bell,1995-Lisonek,2022-Uday}.

The following result was obtained by Baldoni et al. in 2015 \cite{2015-Baldoni-Denumerant} using generating function and technique in computational geometry.
\begin{theo}\label{theo-Baldoni}
Given any fixed integer $k$, there is a polynomial time algorithm to compute the highest $k+1$ degree terms of the quasi-polynomial $E(\a;t)$:
\[
  \sum_{i=0}^k E_{N-i}(\a;t) t^{N-i}.
\]
The coefficients are recovered as step polynomial functions of $t$.
\end{theo}
They also developed a \texttt{Maple} package \texttt{M-Knapsack} and a \texttt{C++} package \texttt{LattE Knapsack}.
Both packages are released as part of the software package \texttt{LattE integrale} \cite{2014-Baldoni-LattE}.
A similar result for simplices was obtained earlier by Barvinok \cite{2006-Barvinok-simplex}. The idea was implemented in another released package \texttt{LattE Top-Ehrhart} \cite{2012-Baldoni-LattETop}.
Currently, \texttt{LattE Knapsack} is the best package for $E(\a;t)$.

Baldoni et al.'s computation of $E(\a; t)$ can be illustrated
by the flow chart in Figure \ref{fig-flowchart}. The notation will be defined later.
\begin{figure}[h]
  \centering
  \includegraphics[width=12cm]{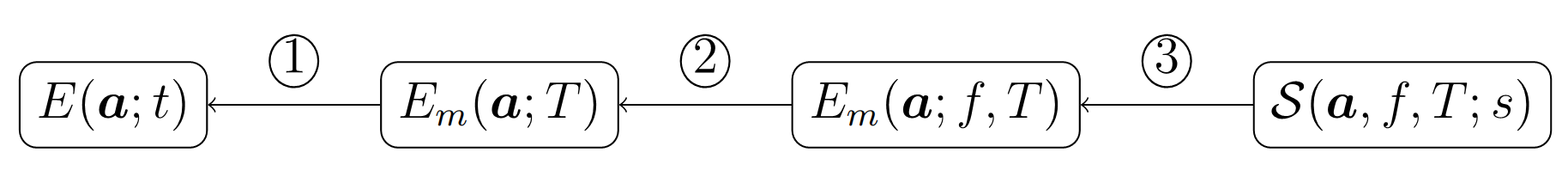}
  \caption{The flow chart of the process of the computation of $E(\a; t)$.}\label{fig-flowchart}
\end{figure}

\begin{enumerate}
  \item[\textcircled{1}] Reduce $E(\a; t)$ to $E_m(\a; T)$ by residue tricks.

  \item[\textcircled{2}] Reduce $E_m(\a; T)$ to $E_m(\a; f, T)$ by the inclusion-exclusion principle.

  \item[\textcircled{3}] Reduce $E_m(\a; f, T)$ to $\S(\a, f, T; s)$ by simple split.
\end{enumerate}
Then the problem is reduced to computing $\S(\a, f, T; s)$, which has a connection with simplicial cone, and hence can be solved using Barvinok's
unimodular cone decomposition.

Denumerant problem was also studied in algebraic combinatorics in terms of constant term manipulation. See, e.g., \cite{2012-Zeilberger,2015-Xin-CTEuclid}.
The concept of cone is known to the combinatorial community, but Barvinok's unimodular cone decomposition was treated as a black box.

Our original motivation of this work is to give a better understanding of Theorem \ref{theo-Baldoni} in constant term language. The most intriguing part of Baldoni et al.'s work is the complicated treatment of $\S(\a, f, T; s)$. The other steps are not hard to follow. The use of inclusion-exclusion principle
in \textcircled{2} is a necessary step for their polynomial time algorithm.

Our main contributions in this paper are as follows.
\begin{enumerate}
  \item Introduce Barvinok's decomposition of simplicial cone into unimodular cones
  as a tool (not a black box) to the combinatorial community.
  Algorithm \ref{alg-barv} is provided and an explicit example is given.

  \item Connect $\S(\a, f, T; s)$ with a simplicial cone by two simple results in constant term field, namely, Lemmas \ref{lem-cone-1} and \ref{lem-cone-2}.

  \item After unimodular cone decomposition, one needs to compute certain limits. We use two results in a recent work \cite{2023-Xin-GTodd} to simplify the computation.

  \item For each $f$, Algorithm \ref{alg-Emf} computes $E_m(\a; f, T)$ simultaneously for different $m$, and hence avoids plenty of repeated computations.

  \item We develop a \texttt{Maple} package \texttt{CT-Knapsack}, which has a significant speed advantage over \texttt{M-Knapsack}.
      See Section \ref{s-experiments} for details.
\end{enumerate}

We give an example to show the result obtained by our algorithm.
\begin{exam}\label{exam-2336}
For $\a = (2, 3, 3, 6)$, our algorithm gives
\begin{align*}
E(\a; t) &= \frac{1}{648} t^3 + \Big(\frac{1}{24} - \frac{\{\frac{2t}{3}\}}{36}\Big) t^2
+ \Big(\frac{13}{36} - \frac{\{\frac{t}{2}\}^2}{6} + \frac{\{\frac{2t}{3}\}^2}{6} - \frac{\{\frac{2t}{3}\}}{2} \Big) t \\
&\quad + \frac{2 \{\frac{t}{2}\}^3}{3} + \Big\{\frac{t}{2}\Big\}^2\Big\{\frac{2t}{3}\Big\} - \frac{\{\frac{2t}{3}\}^3}{3} - \frac{3 \{\frac{t}{2}\}^2}{2} + \frac{3 \{\frac{2t}{3}\}^2}{2} - \frac{\{\frac{t}{2}\}}{6} - \frac{13 \{\frac{2t}{3}\}}{6} + 1,
\end{align*}
where $\{r\} := r - \lfloor r \rfloor$ means the fractional part of $r$ for any $r \in \RR$.
It can be seen that the periods of $E_0(\a; t)$ and $E_1(\a; t)$ are equal to $6\ (=\lcm(\a))$,
and $E_2(\a; t)$ and $E_3(\a; t)$ have periods $3$ and $1$ respectively.
This fits Bell's result.

Taking $t=8$ gives $E(\a; 8)=5$. This says that there are exactly $5$ nonnegative integer solutions to the equation $2 x_1 + 3 x_2 + 3 x_3 + 6 x_4 = 8$.
We list them here:
\begin{align*}
x_1 &= x_4 = 1, x_2 = x_3 =0; & x_1 &= x_2 = x_3 = 1, x_4 = 0; & x_1 &= 1, x_2 = 2, x_3 = x_4 = 0; \\
x_1 &= 1, x_3 = 2, x_2 = x_4 = 0; & x_1 & = 4, x_2 = x_3 = x_4 = 0.
\end{align*}
\end{exam}

The structure of this paper is as follows.
In Section \ref{s-reduction}, we redo the reduction in \textcircled{1} using entirely Laurent series manipulations, and redo the reduction in \textcircled{2} by giving a better way to compute the posets and M\"obius function. This completes the reduction from $E(\a; t)$ to $E_m(\a; f, T)$.
In order to compute $E_m(\a; f, T)$, we introduce three tools for algebraic combinatorics in Section \ref{s-Barvinok}:
i) Barvinok's decomposition for simplicial cones;
ii) LattE's implementation of Barvinok's idea;
iii) a connection between certain constant term and a special kind of simplicial cones.
The computation is done in Section \ref{s-computation} using these tools.
Section \ref{s-experiments} gives a summary of our algorithm and computer experiments.
Section \ref{s-cr} is the concluding remark.

\section{The reduction from $E(\a; t)$ to $E_m(\a; f, T)$}\label{s-reduction}
We first prove the following result.
\begin{theo}\label{theo-ressum}
Suppose $r_1, \dots, r_k$ are positive integers, $b$ is a nonnegative integer and
\[
R(z)=\frac{z^{-b-1}}{(z-\xi_1)^{r_1}\cdots (z-\xi_k)^{r_k}}.
\]
Then
\[
\res_{z=0} R(z) = -\sum_{i=1}^k   \res_{z=\xi_i} R(z).
\]
\end{theo}
\begin{proof}
A well-known result in residue computation asserts that
\[
\res_{z=\infty} R(z) + \res_{z=0} R(z) + \sum_{i=1}^k   \res_{z=\xi_i} R(z) = 0.
\]
The theorem then follows by $\res\limits_{z=\infty} R(z) = 0$, since when expanded as a Laurent series at $z=\infty$,
$(z-\xi_i)^{-1}=z^{-1}(1-\xi_i/z)^{-1}$ has only negative powers in $z$.
\end{proof}

Let $F(\a; z, t)$ be as in \eqref{e-Fofa}. Applying Theorem \ref{theo-ressum} to $F(\a;z,t)$ gives
\[
E(\a;t)= \res_{z=0} F(\a; z,t) =-\sum_{\zeta} \res_{z=\zeta} F(\a; z,t),
\]
where $\zeta$ ranges over all nonzero poles of $F(\a; z, t)$, that is, $\zeta^{a_i}=1$ for some $i$.
For each residue, make the variable substitution $z=\zeta \e^s$, with $\d z=\zeta \e^s \d s$. We obtain
\begin{align*}
E(\a;t)&=-\sum_{\zeta}  \res_{s=0} F(\a; \zeta \e^s, t) \zeta \e^s \\
&= -\sum_{\zeta}  \res_{s=0} \frac{(\zeta \e^s)^{-t}}{\prod_{i=1}^{N+1} (1-\zeta^{a_i} \e^{a_is})}\\
&=  -\sum_{\zeta}  \res_{s=0} \e^{-ts} \frac{\zeta^{-T}}{\prod_{i=1}^{N+1} (1-\zeta^{a_i} \e^{a_is})} \Big|_{T=t} \\
&= -\sum_{m\ge 0} t^m \res_{s=0} \frac{(-s)^{m}}{m!} \sum_{\zeta}\frac{\zeta^{-T}}{\prod_{i=1}^{N+1} (1-\zeta^{a_i} \e^{a_is})} \Big|_{T=t}.
 \end{align*}

Thus we can write
\[
E(\a;t)=\sum_{m= 0}^{N} t^m E_m(\a; T) \big|_{T=t},
\]
where we will show that $E_m(\a;T)=0$ when $m\ge N+1$, and
\[
E_m(\a; T)= -\res_{s=0} \frac{(-s)^{m}}{m!} \sum_{\zeta}\frac{\zeta^{-T}}{\prod_{i=1}^{N+1} (1-\zeta^{a_i} \e^{a_is})}.
\]

For each $\zeta$, $(1-\zeta^{a_i} \e^{a_is})^{-1}$ is either analytic at $s=0$ when $\zeta^{a_i}\neq 1$ or has $s=0$ as a simple pole when $\zeta^{a_i}= 1$.
It follows that the pole order of $\frac{\zeta^{-T}}{\prod_{i=1}^{N+1} (1-\zeta^{a_i} \e^{a_is})}$ is equal to $\# \{i : \zeta^{a_i}=1\}\leq N+1$. Thus
$E_m(\a; T)=0$ when $m \ge N+1$.

To see $E_m(\a; t)=E_m(\a; T)\big|_{T=t}$ is the desired coefficient of the quasi-polynomial $E(\a;t)$, we show that $E_m(\a; t)$ is a periodic function in $t$
with a period $\lcm(\a)$. This follows by the fact that $\zeta^{a_i}=1$ for some $a_i$ and hence $\zeta^{\lcm(\a)}=1$ holds for all $\zeta$.

Next we simplify $E_m(\a; T)$ for a particular $m$. The sum in $E_m(\a; T)$ is only over $\zeta$ satisfying $\# \{i : \zeta^{a_i}=1\}\geq m+1$.
This is equivalent to saying that if $\zeta$ is an $f$-th primitive root of unity, i.e., $\zeta=\exp(2j\pi\sqrt{-1}/f)$ with $\gcd(j,f)=1$,
then $\# \{i : \zeta^{a_i}=1\} = \# \{i : f \mid a_i\}\geq m+1$. In other words, $f$ is a factor of $\gcd(\a_J)$, where
$J=\{j_1,\dots, j_{m+1}\}$ is an $m+1$ subset of $\{1,2,\dots,N+1\}$, and $\a_J=(a_j)_{j\in J}$.

A crucial observation in \cite{2015-Baldoni-Denumerant} is that the number of summands in $E_m(\a; T)$ depends on the $a_j$'s, and may be exponential in the input size
$\sum_{j} \log(a_j)+1$,  while $\gcd(\a_J)$ has at most $2^{N+1}-1$ values. Moreover, if we restrict $|J|\geq m+1$, then $\gcd(\a_J)$ has at most
$\sum_{j\geq m+1} \binom{N+1}{j} =\sum_{j\leq N- m} \binom{N+1}{j}$ values, which is a polynomial in $N$ of degree $k$ when $m=N-k$.
This is why there is a polynomial time algorithm when $k$ is fixed and $N$ is taken as an input.

To carry out the idea, define $\f_m := \{\gcd(\a_J): |J| \ge m+1\}$.
Then we can write
\begin{equation}\label{e-Em-1}
E_m(\a; T)= -\res_{s=0} \frac{(- s)^{m}}{m!} \sum_{\zeta \in \P_m} \frac{\zeta^{-T}}{\prod_{i=1}^{N+1} (1-\zeta^{a_i} \e^{a_is})},
\end{equation}
where $\P_m = \{\zeta : \zeta^f =1 \text{ for some } f \in \f_m\}$.

The following result reduces the computation of $E_m(\a; T)$ to that of $E_m(\a; f, T)$. It was first obtained in \cite{2015-Baldoni-Denumerant} by the inclusion-exclusion principle. Our statement is slightly different.
\begin{prop}\label{prop-EmeqsummuEmf}
Let $\a=(a_1,\dots,a_{N+1})$ be a positive integer sequence and $E_m(\a; T)$ be as in \eqref{e-Em-1}. We have
\begin{equation}\label{e-Em-2}
E_m(\a; T) = \sum_{f \in \f_m} \mu_m(f) E_m(\a; f, T),
\end{equation}
where
\begin{equation}\label{e-Mobi}
\mu_m(f) = 1-\sum_{f'\in \f_m : f \mid f', f'\neq f}\mu_m(f'),
\end{equation}
\begin{equation}\label{e-Emf-1}
E_m(\a; f, T) = \res_{s=0}- \frac{ (-s)^{m}}{m!} \sum_{\zeta:\zeta^f=1} \frac{\zeta^{-T}}{\prod_{i=1}^{N+1} (1-\zeta^{a_i} \e^{a_is})}.
\end{equation}
In particular, $\mu_m(f)=1$ if $\{f'\in \f_m : f \mid f', f' \neq f\}=\varnothing$.
\end{prop}
\begin{proof}
Denote by $\P^{(f)} := \{\zeta : \zeta^f =1\}$.
Then $\P_m = \cup_{f \in \f_m} \P^{(f)}$. The set $P_m=\{\P^{(f)} : f \in \f_m \}$ forms a poset (still denoted by $P_m$) with respect to reverse inclusion.
That is, $\P^{(f')} \preceq \P^{(f)}$ if $\P^{(f')} \supseteq \P^{(f)}$, i.e., $f \mid f'$.
We add the minimum element $\hat{0}$ into $P_m$ and call the new poset $\hat{P}_m$.
Then by the inclusion-exclusion principle, we can write the indicator function of $\P_m$ as a linear combination of the indicator functions of $\P^{(f)}$'s:
\[
[\P_m] = \sum_{f \in \f_m} \mu_m(f) [\P^{(f)}],
\]
where $\mu_m(f) := - \mu'_m(\hat{0}, \P^{(f)})$ and $\mu'_m(x, y)$ is the standard M\"obius function for the poset $\hat{P}_m$.
That is, $\mu'_m(x, x) = 1, \forall x \in \hat{P}_m$ and $\mu'_m(x, y) = -\sum_{x \preceq x' \prec y} \mu'_m(x, x'), \forall x \preceq y$.
Simple computation gives that
\begin{align*}
\mu_m(f) &= - \mu'_m(\hat{0}, \P^{(f)}) \\
&= \sum_{\hat{0} \preceq \P^{(f')} \prec \P^{(f)}} \mu'_m(\hat{0}, \P^{(f')}) \\
&= \mu'_m(\hat{0}, \hat{0}) - \sum_{\hat{0} \prec \P^{(f')} \prec \P^{(f)}} -\mu'_m(\hat{0}, \P^{(f')}) \\
&= 1 - \sum_{\hat{0} \prec \P^{(f')} \prec \P^{(f)}} \mu_m(f')\\
&= 1-\sum_{f'\in \f_m : f \mid f', f' \neq f}\mu_m(f').
\end{align*}
Then the lemma follows.
\end{proof}

The basic building block is $E_m(\a; f, T)$, which will be computed in Section \ref{s-computation} using tools in Section \ref{s-Barvinok}.
From \eqref{e-Em-2}, we see that
$E_m(\a; f, T)$ is not needed if $\mu_m(f) = 0$. This invokes us to define
\[
 \m(f) := \{m \in \m : \mu_m(f) \neq 0\},
\]
where $\m$ denotes the subscript set of $E_m(\a; T)$ that we want to compute.
Indeed, for each $f$, we compute $E_m(\a;f,T)$ simultaneously for all $m\in \m(f)$ (See Algorithm \ref{alg-Emf}).

To apply Proposition \ref{prop-EmeqsummuEmf}, we describe a better way to compute $\f_m$ and $\mu_m(f)$ for each $f\in \f_m$. In later computation, we also need the multi-set  $\a(f) = \{ a_i: f\mid a_i\}^*$, where we use the superscript $^*$ to denote multi-set.

\begin{algorithm}[H]
\DontPrintSemicolon
\KwInput{$\a$.}
\KwOutput{$S = \cup_{m\ge 0} \f_m$ and $L = \{(f, \a(f)): f \in S\}$.}

Assume $\a = (n_1\cdot a_1, \dots, n_k \cdot a_k)$, where $n_i \cdot a_i$ means $a_i$ occurs $n_i$ times, $a_1 < \cdots < a_k$.

\If{$k = 1$}
{
    \Return $S=\{a_1\}$ and $L = \{(a_1, \{n_1 \cdot a_1\}^*)\}$.
}
\Else
{

    Set $\a' = (a_1, \dots, a_k)$, $S = \{a_1\}$ and $L = \{(a_1, \{a_1\})\}$.

    \For{ $i$ \rm{from} $2$ \rm{to} $k$}
    {
        $tS := S$.  \qquad  \# (Denoted $S^{(i-1)}$.)

        \For{$f \in tS$}
        {
            Compute $tf := \gcd(a_i, f)$.

            \If{$tf \in S$}
            {
                $\a(tf) := \a(tf) \cup \{a_i\}$.
            }
            \Else
            {
                $\a(tf) := \a(f) \cup \{a_i\}$;

                $S := S \cup \{tf\}$ and $L := L \cup \{(tf, \a(tf))\}$.
            }
        }

        $S := S \cup \{a_i\}$ and $L := L \cup \{(a_i, \{a_i\})\}$.
    }
}

Update $\a(f)$ by replacing $a_i$ by $n_i \cdot a_i$ for all $f \in S$.

\Return $S$ and $L$.

\caption{Algorithm \texttt{fset}}\label{alg-fset}
\end{algorithm}

\begin{rem}
From the outputs of Algorithm \texttt{fset}, we can easily obtain $\f_m = \{f \in S : \#\a(f) \ge m+1\}$ for all $m$.
Algorithm \texttt{fset} only performs $|S^{(1)}|+\cdots + |S^{(k-1)}|$ $\gcd$ computations. This avoids plenty of repeated $\gcd$ computations.
In the worst case, $k=N+1$, $|S^{(i)}|= 2^i-1$ and we still need $2^{N+1}-N-2$ $\gcd$ computations.
The algorithm is not efficient when we only need $\gcd (\a_{J})$ for all $|J|\ge N-k$ when $k$ is fixed and $N$ is large.
\end{rem}

\begin{exam}\label{ex-fset}
Input $\a = (2,3,3,6)$.

1. $\a = (1 \cdot 2, 2 \cdot 3, 1 \cdot 6)$. Set $\a' = (2,3,6)$, $S = \{2\}$, $L = \{(2, \{2\})\}$.

2. $a_2 = 3$, $tS := \{2\}$.

(2.1) Compute $\gcd(3,2) = 1 \notin S$ and update $S := \{1, 2\}$ and $L := \{(1, \{2,3\}), (2, \{2\})\}$;

(2.2) Update $S := \{1, 2, 3\}$ and  $L := \{(1, \{2,3\}), (2, \{2\}), (3, \{3\})\}$.

3. $a_3 = 6$, $tS := \{1, 2, 3\}$.

(3.1) Compute $\gcd(6,1) = 1 \in S$ and update $\a(1) := \{2, 3, 6\}$;

(3.2) Compute $\gcd(6,2) = 2 \in S$ and update $\a(2) := \{2, 6\}$;

(3.3) Compute $\gcd(6,3) = 3 \in S$ and update $\a(3) := \{3, 6\}$;

(3.4) Update $S := \{1, 2, 3, 6\}$ and $L := \{(1, \{2, 3, 6\}), (2, \{2,6\}), (3, \{3, 6\}), (6, \{6\})\}$.

4. Update $\a(1) := \{2, 3, 3, 6\}^*$, $\a(2) := \{2, 6\}^*$, $\a(3) := \{3, 3, 6\}^*$ and $\a(6) := \{6\}^*$.

5. Output $S=\{1,2,3,6\}$ and $L = \{(1, \{2, 3, 3, 6\}^*), (2, \{2, 6\}^*), (3, \{3, 3, 6\}^*), (6, \{6\}^*)\}$.

By the outputs, we have $\f_0 = \{1,2,3,6\}$, $\f_1 = \{1,2,3\}$, $\f_2 = \{1,3\}$ and $\f_3 = \{1\}$.
\end{exam}

The following algorithm is used to compute $\mu_m(f)$ for each $f \in \f_m$.

\begin{algorithm}
\DontPrintSemicolon
\KwInput{$\f_m$.}
\KwOutput{$\{(f, \mu_m(f)): f \in \f_m, \mu_m(f) \neq 0\}$.}

Compute $\mu_m(f)$ from large to small for $f \in \f_m$ by the formula \eqref{e-Mobi}.
The computation order from large to small ensures that all $\mu_m(f')$ required have been obtained when computing $\mu_m(f)$.

\Return $\{(f, \mu_m(f)): f \in \f_m, \mu_m(f) \neq 0\}$.

\caption{Algorithm \texttt{M\"obius}}\label{alg-Morb}
\end{algorithm}

\begin{exam}\label{ex-Morb}
Continue Example \ref{ex-fset}, we apply Algorithm \texttt{M\"obius} to $\f_0 = \{1,2,3,6\}$.
The detailed computations are as follows:
\begin{align*}
& \mu_0(6) = 1; &
& \mu_0(3) = 1 - \mu_0(6) =0; \\
& \mu_0(2) = 1 - \mu_0(6) =0; &
& \mu_0(1) = 1 - \mu_0(2) - \mu_0(3) - \mu_0(6) = 0.
\end{align*}
Finally, output $\{(6, 1)\}$. Thus $E_0(\a; T) = E_0(\a; 6, T)$.
Similarly, by applying Algorithm \texttt{M\"obius} to $\f_1=\{1,2,3\}$, $\f_2=\{1,3\}$ and $\f_3 = \{1\}$ respectively, we can obtain
\begin{align*}
E_1(\a; T) &= -E_1(\a; 1, T) + E_1(\a; 2, T) + E_1(\a; 3, T), & E_2(\a; T) &= E_2(\a; 3, T),\\
E_3(\a; T) &= E_3(\a; 1, T).
\end{align*}
Moreover, if $\m=\{0,1,2,3\}$, then $\m(1)=\{1,3\}$, $\m(2)=\{1\}$, $\m(3)=\{1,2\}$, $\m(6)=\{0\}$.
\end{exam}

\section{Main tools in the computation of $E_m(\a; f, T)$}\label{s-Barvinok}
This section introduces several tools used in our computation of $E_m(\a; f, T)$.
We first briefly introduce Barvinok's unimodular cone decomposition and its implementation by LattE.
The well-known LLL's algorithm \cite{1982-LLL} plays an important role. Then we introduce
a constant term concept and its connection with simplicial cone.

This section has its own notation system, for instance $A^T$ means the transpose of the matrix $A$.
We hope to present the materials in the context of algebraic combinatorial theory.

\subsection{Barvinok's unimodular cone decomposition}
We follow notation in \cite{2004-LattE}, but only introduce necessary concepts.

Let $A=(\alpha_1\mid \alpha_2\mid \dots\mid \alpha_d)$ be a $d\times d$ nonsingular rational matrix.
Then its column vectors generate a lattice and a cone defined by
\begin{align*}
 L(A)&=L(\alpha_1 \mid \alpha_2 \mid \dots \mid \alpha_d)=A\ZZ^d = \{k_1\alpha_1+\cdots+k_d \alpha_d: k_i\in \ZZ, \ i=1,2,\dots, d\},\\
 C(A)&=C(\alpha_1 \mid \alpha_2 \mid \dots \mid \alpha_d)= \{k_1\alpha_1+\cdots+k_d \alpha_d: k_i\in \RR_{\geq 0}, \ i=1,2,\dots, d\}.
\end{align*}
The $\alpha_1,\dots, \alpha_d$ is also called a lattice basis of $L(A)$. The LLL's algorithm
computes a reduced lattice basis
of $L(A)$. That is, it outputs a new $d\times d$ matrix $A'=(\alpha_1'\mid \alpha_2'\mid \dots\mid \alpha_d')$ with $L(A)=L(A')$.
The column vectors $\alpha_1', \alpha_2', \dots, \alpha_d'$
are short and nearly orthogonal to each other, and each $\alpha_i'$ is an approximation of the shortest vector in the lattice, in terms of Euclidean length.

LLL's algorithm has many applications due to the nice properties of the reduced basis. We will use it to
\emph{find the smallest vector in a lattice}. The \emph{smallest} vector $\beta=(b_1,\dots, b_d)$ in the lattice $A\ZZ^d$ is the nonzero vector
with the minimum infinity norm $||\beta||_\infty=\max_{i} |b_i|$.
Dyer and Kannan \cite{1997-Baralg} observed that the $\alpha_i'$'s are also approximations of the smallest vector $\beta \in A\ZZ^d$, and $\beta$ can be searched over all $A' \gamma$ for integral $\gamma$ with small entries. This search is referred to as the \emph{enumeration step}.

Now we turn to the cone $C(A)$. It is called \emph{simplicial} because its generators are linearly independent.
By definition, the cone is kept same if we replace $\alpha_i$ by $\bar{\alpha}_i=c_i \alpha_i$ for any $c_i\in \RR_{>0}$.
We choose $\bar{\alpha}_i$ to be the unique primitive vector, i.e., $\bar{\alpha}_i\in \ZZ^d$
and $\gcd (\bar{\alpha}_i) = 1$. Denote by $\bar{A} = (\bar{\alpha}_1\mid \bar{\alpha}_2\mid \dots\mid \bar{\alpha}_d)$, and by $\ind (\K) = |\det(\bar A)|$ the \emph{index} of the simplicial cone $\K=C(A)$.

For any set $S \subset \RR^n$, denote the multivariate generating function of $S$ by
\[
\sigma_{S}( \y)=\sum_{\alpha \in S \cap \ZZ^n } \y^\alpha.
\]
A well-known result of Stanley \cite{book-Stanley-EC1} states that
\[
\sigma_{\K}(\y)= \frac{\sigma_{\Pi(\K)}(\y)}{\prod_{j=1}^{d} (1-\y^{\bar\alpha_{j}})},
\]
where  $\Pi(\K)=\{ k_1 \bar\alpha_{1} +\cdots+ k_{d} \bar\alpha_d : 0 \le k_1,\dots,k_d < 1 \}$ is called the \emph{fundamental parallelepiped} of the cone $\K$.
Furthermore, $\ind (\K)$ gives the number of lattice points in $\Pi(\K)$. In particular, if $\bar A$ is unimodular, i.e., $\det( \bar A)=\pm 1$,
then $\K$ is called a \emph{unimodular cone} and $\sigma_{\K}(\y)$ has numerator $1$.

When $\K$ is shifted by a vector $v$, we obtain a \emph{vertex cone}
\[
\K^v = v + C(A) = \{v + k_1\alpha_1+\cdots+k_d \alpha_d: k_i\in \RR_{\geq 0}, \ i=1,2,\dots, d\}.
\]
There is a similar result of $\sigma_{\K^v}(\y)$, but we only need the formula for unimodular $\K$.
\begin{prop}\label{prop-VCone}
Suppose $\K=C(\bar{\alpha}_1\mid \bar{\alpha}_2\mid \dots\mid \bar{\alpha}_d)$ is a unimodular cone in $\RR^d$.
If $v$ is uniquely written as $v=\ell_1\bar\alpha_1+\cdots+\ell_d\bar\alpha_d$,
then \[
\sigma_{\K^v}( \y)=\sum_{\alpha \in \K^v \cap \ZZ^d } \y^\alpha = \frac{\y^{\sum_{j=1}^d \lceil \ell_j \rceil \bar\alpha_{j}}}{\prod_{j=1}^{d} (1-\y^{\bar\alpha_{j}})}.
\]
\end{prop}
We are not able to find a reference of this result, but the idea is simply illustrated by Figure \ref{fig-addvertex} for the $d=2$ case.
In Figure \ref{fig-addvertex}, we have depicted the cone $\K = C(\bar\alpha_1 \mid \bar\alpha_2)$ and the corresponding lattice.
For a vertex $v$ uniquely written as $v = \ell_1 \bar\alpha_1 + \ell_2 \bar\alpha_2$, the shifted cone $\K^v =  \{ v + k_1 \bar \alpha_1 + k_2 \bar \alpha_2 : k_1,k_2 \in \RR_{\ge 0} \}$ is shown as the shaded region.
It is clear that
\[
\K^v  \cap \ZZ^2 = \{\lceil \ell_1 \rceil \bar\alpha_1 + \lceil \ell_2 \rceil \bar\alpha_2 + k_1 \bar \alpha_1 + k_2 \bar \alpha_2: k_1,k_2 \in \ZZ_{\ge 0}\}.
\]
\begin{figure}[h]
  \centering
  \includegraphics[width=4cm]{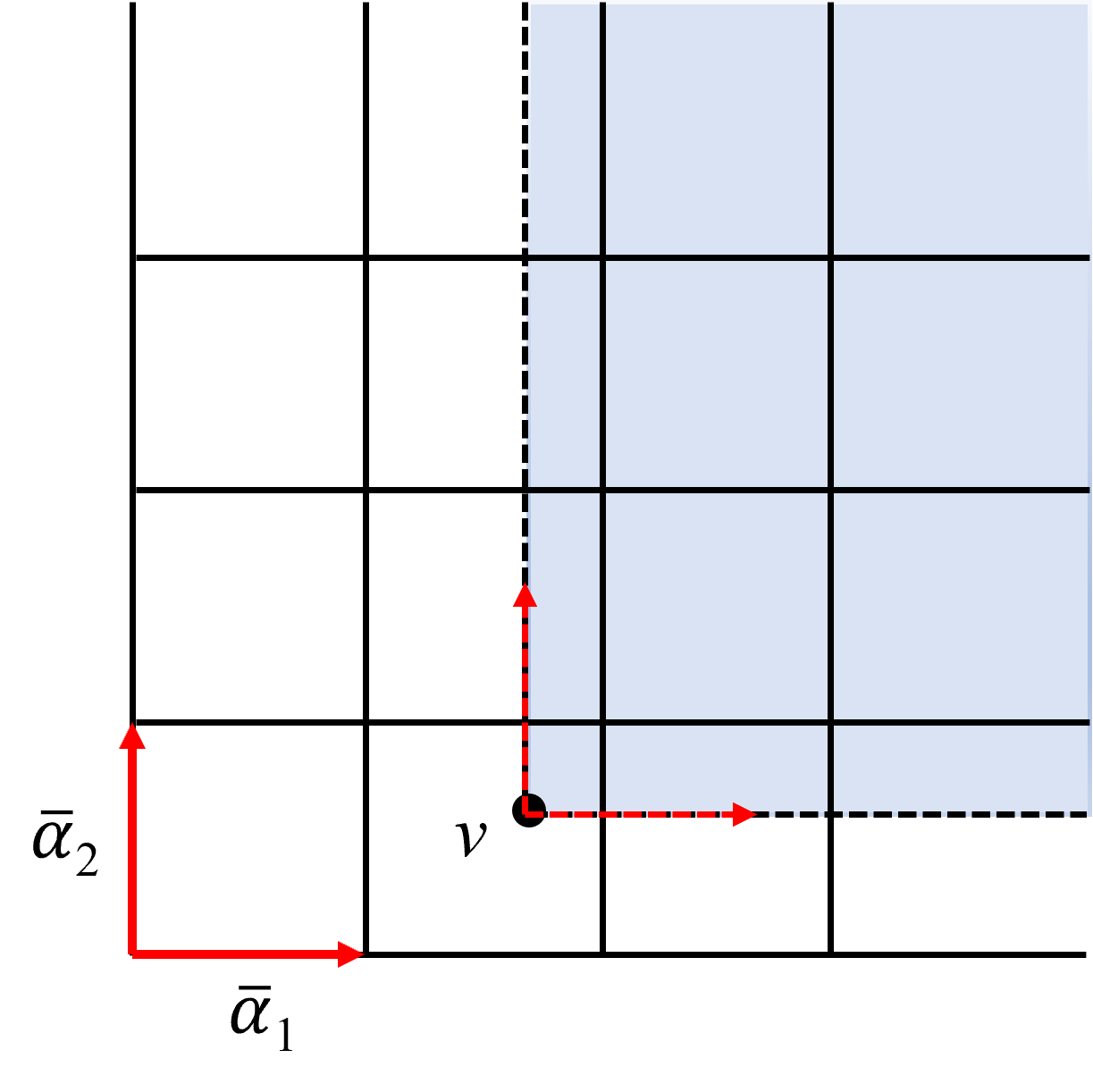}
  \caption{The illustration of $\K^v \cap \ZZ^d$ in the case when $d=2$.}\label{fig-addvertex}
\end{figure}

When $\ind (\K)$ is large, the traditional encoding of $\sigma_{\K}(\y)$ cannot be handled by computer. As early as 1994, Barvinok gave a solution in \cite{1994-Barvinok}. We state the translated version (combining with Proposition \ref{prop-VCone}) as follows.
\begin{theo} \label{theo-polytime-barv}
Fix the dimension $d$. There is a polynomial time algorithm which, for given rational cone $\K=C(A) \subset \RR^d$, computes unimodular cones $\K_i=C(A_i)$, $i\in I=\{1,2,\dots, N\}$ and numbers $\epsilon_i\in \{ 1,-1\}$ such that for any $v\in \RR^d$,
\[
\sigma_{\K^v}(\y)= \sum_{i\in I} \epsilon_i \sigma_{\K_i^v}(\y)=\sum_{i\in I} \epsilon_i
\frac{\y^{\alpha_{i0}}}{\prod_{j=1}^{d} (1-\y^{\alpha_{ij}})}.
\]
The sum is ``short" or polynomial in size, meaning that the number of summands $N=|I|$ is bounded by a polynomial in $\log (\ind (\K))$.
\end{theo}

\subsection{LattE's implementation of Barvinok's decomposition}\label{ss-LattEBarv}
Barvinok's decomposition of $\K=C(\bar A)$ is based on a nonzero lattice point $\gamma$ in the closed parallelepiped
\[
\Gamma =\{k_1\bar \alpha_1+ \cdots+k_d\bar \alpha_d: ||(k_1,\dots, k_d)||_\infty \leq (\ind(\K))^{-\frac{1}{d}} \}.
\]
The existence of such $\gamma$ is guaranteed by  Minkowski's First Theorem \cite{book-Minkowski} because $\Gamma \subset \RR^d$ is a centrally symmetric convex body with volume $\geq 2^d$.

The $\gamma$ is not easy to find. It is sufficient to find a smallest vector $\beta=(k_1,\dots, k_d)^T$
 in $\bar{A}^{-1}\ZZ^d$ and then set $\gamma=\bar A\beta$. The smallest vector would produce better decomposition, but the enumeration step is costly.
LattE simply computes the reduced basis of $\bar{A}^{-1}\ZZ^d$ and set $\beta$ to be the smallest vector among them. If the condition
$||\beta||_\infty \leq (\ind(\K))^{-\frac{1}{d}}$ holds true, then $\gamma \in \Gamma$ as desired; otherwise, one needs the costly enumeration step to find $\beta$.
LattE reported that the enumeration step is rarely used in practice.

Once $\gamma$ (which has to be primitive) is found. Then one can decompose the cone $\K=C(\bar A)$ into a signed combination of the cones $\K_i=C(\bar A_i)$, where
\[
\bar A_i=(\bar \alpha_1\mid \cdots \mid \bar \alpha_{i-1} \mid \gamma \mid\bar \alpha_{i+1}\mid \cdots \mid \bar \alpha_d),
\]
and $\ind(\K_i)=|\det(\bar A_i)|=|k_i \det(\bar A)|\leq \ind(\K)^{\frac{d-1}{d}}.$
The signs depend on the position of $\gamma$ with respect to interior or exterior of $\K$. Note that
if $\gamma$ and $\alpha_1,\dots,\alpha_d$
do not belong to an open half-space, then we shall use $-\gamma$ instead (see Lemma 5.2 \cite{1994-Barvinok}).
Note also that in Barvinok's original work, one also needs to keep tract of (plenty of) lower dimensional cones to decompose $\sigma_{\K}$.

Brion's \cite{1988-Brion-lowerdim} polarization trick allows us to ignore lower dimensional cones.
Denote by $\K^*:=\{\alpha \in \RR^d : \langle \alpha, \beta \rangle \ge 0 \text{ for each } \beta \in \K\}$ the dual cone of $\K$ \cite{1999-Barvinok}.
We only need the following four properties:

i) The dual is an involution, i.e., $(\K^*)^*=\K$;

ii) If $\dim(\K) <d$, then its dual cone $\K^*$ contains a line so that $\sigma_{\K^*}$ is treated as $0$;

iii) If $\K=C(A)$ is full dimensional, then $\K^*$ is generated by the columns of $(A^{-1})^T$;

iv) $\sigma_{\K^*}=\sum_{i} \epsilon_i \sigma_{\K_i^*} \Leftrightarrow \sigma_{\K}=\sum_{i} \epsilon_i \sigma_{\K_i}$.

Thus we can first find a decomposition of $\K^*$ (ignore low dimensional cones) and then dual back to obtain a decomposition of $\K$.

For clarity, we give the pseudocode of the algorithm, which modifies Algorithm 5 in \cite{2004-LattE}.
For a matrix $A$, we shall abbreviate cone $C(A)$ as $A$. This will cause no ambiguity.

\begin{algorithm}[H]\label{alg-barv}
\DontPrintSemicolon
\KwInput{A $d\times d$ nonsingular integer matrix $A$ defining $\K=C(A)$.}
\KwOutput{A list of $(\epsilon_i,A_i)$ as described in Theorem \ref{theo-polytime-barv} in the $v=\mathbf{0}$ case.}

\If{$\bar A$ is unimodular}
{
\Return $[(1,\bar A)]$.
}
\Else
{
    Compute $B$ with primitive column vectors such that $\K^*=C(B)$.

    Set two lists $\Uni=[~]$ and $\NUni=[(1,B)]$.

    \While{$\NUni\neq [~]$}
    {
        Take the last element $(\epsilon_B, B) \in \NUni$ and remove $(\epsilon_B, B)$ from $\NUni$.

        Compute the smallest vector $\beta=(k_1,\dots, k_d)^T$ in $L(B^{-1})$.
        If $k_i\leq 0$ for all $i$, then set $\beta := -\beta$.
        Set $\gamma = B\beta$.

        \For{$i$ \rm{from} $1$ \rm{to} $d$}
        {
            \If{$k_i \neq 0$}
            {
                Set  $B_i$ to be obtained from $B$ by replacing the $i$-th column by $\gamma$.

                \If{$k_i \det(B)=\pm 1$}
                {
                    Compute $A_i=(B_i^{-1})^T$ and add $(\sgn(k_i)\cdot \epsilon_B, A_i)$ to $\Uni$.
                }
                \Else
                {
                    Add $(\sgn (k_i)\cdot\epsilon_B, B_i)$ to $\NUni$.
                }
            }
        }
    }
}
\Return $\Uni$.

\caption{Barvinok's unimodular cone decomposition}
\end{algorithm}

\subsection{An example by a denumerant cone}\label{ss-exam}
We are interested with a special type of simplicial cone, called \emph{denumerant cone} in \cite{2023-Xin-knapsackcone}, denoted
\begin{equation*}\label{e-def-dcone}
\D_j(a_1, \dots, a_n) := \Big\{ \sum_{i=1,i \neq j}^n k_i (-a_i e_j + a_j e_i) : k_i \in \RR_{\ge 0}\Big\},
\end{equation*}
where $a_1,\dots, a_n$ are positive integers with $\gcd(a_1,\dots,a_n) = 1$, and $e_i$ is the standard $i$-th unit vector for each $i$.
It is also called a \emph{knapsack cone} since it arises from the knapsack problem $a_1x_1+\cdots +a_n x_n=b$.
We focus on $\D_1(a_1,\dots,a_n) = C(H)$, where
\[
H = \left(\begin{array}{cccc}
-a_2 & -a_3 & \cdots & -a_n \\
a_1 & 0 & \cdots & 0 \\
0 & a_1 & \cdots & 0 \\
\vdots & \vdots & \ddots & \vdots \\
0 & 0 & \cdots & a_1
\end{array}\right).
\]
This is not a full dimensional cone and the columns of $H$ may not be primitive.
To apply Theorem \ref{theo-polytime-barv}, we need to transform $C(H)$ into a full dimensional cone.
We need a lattice basis of $L(H)$. Such a basis can be obtained by Hermit normal form.
Here we use Smith normal form, say $ V^{-1} (a_1,\dots,a_n) U^{-1}=(1,0,\dots,0)$, where $U$ and $V$ are unimodular matrices.
Thus
\[
(a_1,\dots,a_n)\alpha = 0 \Leftrightarrow V(1,0,\dots,0)U\alpha = 0  \Leftrightarrow (1,0,\dots,0)U\alpha = 0.
\]
Clearly the unit vectors $e_2, \dots, e_n$ form a $\ZZ$-basis of the null space of $(1,0,\dots,0)$. It follows that
$U^{-1}e_i \ (2\le i \le n)$ form a $\ZZ$-basis of the null space of $(a_1,\dots,a_n)$, i.e., a lattice basis of $L(H)$.
Moreover, we have
 $U H= \left(\begin{array}{c} \mathbf{0}^T \\ H' \end{array}\right)$.
Thus we have transformed $C(H)$ into the full dimensional cone
$C(H')$. Then unimodular transformation $U$ gives rise the isomorphism from $C(H)$ to $C(H')$.

Applying Theorem \ref{theo-polytime-barv} to $C(H')$ gives a unimodular cone decomposition:
\[
\sigma_{C(H')}(\y) = \sum_{i \in I} \epsilon_i \sigma_{C(B_i)}(\y),
\]
where $B_i$ is unimodular for each $i \in I$.
Then we have
\[
\sigma_{C(H)}(\y) = \sum_{i \in I} \epsilon_i \sigma_{C(A_i)}(\y),
\]
where $A_i=U^{-1} \left(\begin{array}{c} \mathbf{0}^T \\ B_i \end{array}\right)$ for each $i \in I$.

For the sake of clarity, we illustrate by the following example.
\begin{exam} \label{ex-barv}
Consider the denumerant cone $\D_{1}(10, 11, 5, 17) = C(H)$, where
\[
H=\left( \begin {array}{ccc}
        -11 & -5 & -17\\
        10 & 0 & 0\\
        0 & 10 & 0\\
        0 & 0 & 10
        \end {array}
        \right).
\]
Taking unimodular matrix
 $U= \left( \begin{array}{cccc}
     10 & 11 & 5 & 17\\
     -1 & -1 & -1 & -2\\
      0 & 0 & 1 & 0\\
      0 & 0 & 0 &1
    \end{array} \right)$,
we have $UH = \left(\begin{array}{c} \mathbf{0}^T \\ A \end{array} \right)$,
where $A= \left( \begin {array}{ccc}
        1 & -5 & -3\\
        0 & 10 & 0\\
        0 & 0 & 10
    \end {array}\right)$.
Now we apply Algorithm \ref{alg-barv} to $A$.

(1) $\bar A$ is not unimodular. The dual cone of $A$ is
 $B
 = \left( \begin {array}{ccc}
     10 & 0 & 0\\
     5 & 1 & 0\\
     3 & 0 & 1
 \end {array} \right)$.

Set two lists $\Uni=[~]$ and $\NUni=[(1,B)]$.
Use LLL's algorithm to find the reduced basis of $B^{-1}$, denoted
$B'=\left( \begin {array}{ccc}
2/5 & -1/5 & 1/10 \\
0 & 0 & -1/2 \\
-1/5 & -2/5 & -3/10
\end {array} \right)$.
Both the first and second columns of $B'$ can be chosen as the smallest vector $\beta$.
Here we take $\beta =(-1/5, 0, -2/5)^T$. Since $\beta$ has no positive entries, we set $\beta = (1/5, 0, 2/5)^T$ and then $\gamma = B\beta =(2, 1, 1)^T.$
Replacing the $i$-th column of $B$ by $\gamma$ for $i=1, 3$ respectively gives two cones:
\[
B_1=\left( \begin {array}{ccc}
     2 & 0 & 0\\
     1 & 1 & 0\\
     1 & 0 & 1
 \end {array} \right),
B_3=\left( \begin {array}{ccc}
     10 & 0 & 2\\
     5 & 1 & 1\\
     3 & 0 & 1
 \end {array} \right)
\]
with signs $\epsilon_{B_1} = \epsilon_{B_3} = 1$.
Both of them are not unimodular and we have
\[
\NUni=[(1,B_1), (1,B_3)], \;\Uni = [~].
\]

The remaining steps are similar, we only show the results.

(2) Now $\NUni=[(1,B_1), (1,B_3)]$. Taking $(1,B_3)$ gives three cones:
\[
B_{3,1}= \left( \begin {array}{ccc}
     -3 & 0 & 2\\
     -1 & 1 & 1\\
     -1 & 0 & 1
\end {array} \right),
B_{3,2} =\left( \begin {array}{ccc}
     10 & -3 & 2\\
     5 & -1 & 1\\
     3 & -1 & 1
\end {array} \right),
B_{3,3} = \left( \begin {array}{ccc}
     10 & 0 & -3\\
     5 & 1 & -1\\
     3 & 0 & -1
\end {array}\right)
\]
with signs $\epsilon_{B_{3,1}} = \epsilon_{B_{3,3}} = -1$, $\epsilon_{B_{3,2}} = 1$.
Both $B_{3,1}$ and $B_{3,3}$ are unimodular but $B_{3,2}$ is not.
Then we have
\[
\NUni = [(1, B_1), (1, B_{3,2})], \;
\Uni = [(-1, A_1), (-1, A_2)],
\]
where $A_1 = (B_{3,1}^{-1})^T$, $A_2 = (B_{3,3}^{-1})^T$.

(3) Now $\NUni = [(1, B_1), (1, B_{3,2})]$. Taking $(1, B_{3,2})$ gives two cones:
\[
B_{3,2,1}= \left( \begin {array}{ccc}
     6 & -3 & 2\\
     3 & -1 & 1\\
     2 & -1 & 1
\end {array} \right),
B_{3,2,3} =\left( \begin {array}{ccc}
     10 & -3 & 6\\
     5 & -1 & 3\\
     3 & -1 & 2
\end {array} \right)
\]
with signs $\epsilon_{B_{3,2,1}} = \epsilon_{B_{3,2,3}} = 1$.
Both of them are unimodular and we have
\[
\NUni = [(1, B_1)],\;
\Uni=[(-1, A_1), (-1, A_2), (1, A_3), (1, A_4)],
\]
where $A_3 = (B_{3,2,1}^{-1})^T$, $A_4 = (B_{3,2,3}^{-1})^T$.

(4) Now $\NUni = [(1, B_1)]$. Taking $(1,B_1)$ gives three cones:
\[
B_{1,1}= \left( \begin {array}{ccc}
     1 & 0 & 0\\
     0 & 1 & 0\\
     1 & 0 & 1
\end {array} \right),
B_{1,2} =\left( \begin {array}{ccc}
     2 & 1 & 0\\
     1 & 0 & 0\\
     1 & 1 & 1
\end {array} \right),
B_{1,3} = \left( \begin {array}{ccc}
     2 & 0 & 1\\
     1 & 1 & 0\\
     1 & 0 & 1
\end {array}\right)
\]
with signs $\epsilon_{B_{1,1}} = \epsilon_{B_{1,3}} = 1$, $\epsilon_{B_{1,2}} = -1$.
All of them are unimodular and we have
\[
\NUni = [~],\;
\Uni=[(-1, A_1), (-1, A_2), (1, A_3), (1, A_4), (1, A_5), (-1, A_6), (1, A_7)],
\]
where $A_5 = (B_{1,1}^{-1})^T$, $A_6 = (B_{1,2}^{-1})^T$, $A_7 = (B_{1,3}^{-1})^T$.

(5) Now $\NUni = [~]$ and the loop ends. Output $\Uni$.

Finally, for elements in $\Uni$, let
\[
D_i = U^{-1} \left( \begin{array}{c} \mathbf{0}^T \\ A_i \end{array} \right), \quad
\epsilon_i = \epsilon_{A_i} \quad
\text{ for } \quad 1\le i \le 7.
\]
Then the list of elements $(\epsilon_i, D_i), \ 1\le i \le 7$ is the decomposition of  $\D_{1}(10, 11, 5, 17)$.
\end{exam}

\subsection{A constant term concept and denumerant cone}
We need a basic constant term concept in \cite{2015-Xin-CTEuclid}.
Suppose an Elliott rational function $E=E(\lambda)$ is written in the following form:
\[
E= \frac{L(\lambda)}{\prod_{i=1}^{n} (1 - u_{i} \lambda^{a_i})},
\]
where $L(\lambda)$ is a Laurent polynomial, $u_i$'s are free of $\lambda$ and $a_i$'s are positive integers.
Assume we have the following partial fraction decomposition of $E$ with respect to $\lambda$:
\[
E = P(\lambda) + \frac{p(\lambda)}{\lambda^k} + \sum_{i=1}^{n} \frac{A_i(\lambda)}{1 - u_{i} \lambda^{a_i}},
\]
where $P(\lambda)$, $p(\lambda)$ and the $A_i(\lambda)$'s are all polynomials, $\deg p(\lambda) < k$ and $\deg A_i(\lambda) < a_i$ for all $i$.
We denote by
\begin{equation}\label{e-siglecontri}
\CT_\lambda \frac{1}{\underline{1 - u_{i_0} \lambda^{a_{i_0}}}} \frac{L(\lambda)}{\prod_{i=1, i \neq i_0}^{n} (1 - u_{i} \lambda^{a_i})} := A_{i_0}(0).
\end{equation}
This is a basic building block in the \texttt{CTEuclid} algorithm in \cite{2015-Xin-CTEuclid}.

We need the following two results. See \cite{2023-Xin-CTCone} for the detailed proofs.

\begin{lem}\label{lem-cone-1}
Suppose $R(\lambda)$ is a rational function and $f$ is a positive integer.
If $R(\zeta)$ exist for all $\zeta$ satisfying $\zeta^f=1$, then
\[
\CT_\lambda \frac{1}{\underline{1-\lambda^f}} R(\lambda) = \frac{1}{f}\sum_{\zeta: \zeta^f=1} R(\zeta).
\]
\end{lem}

\begin{lem}\label{lem-cone-2}
Let $a_1,\dots, a_d$ and $f$ be positive integers satisfying $\gcd(f, a_1, \dots, a_d)=1$, and $b$ be an integer.
Then the constant term
\begin{equation}\label{e-cone-2}
  \CT_\lambda \frac{1}{\underline{1- \lambda^f y_0}} \frac{\lambda^{-b}}{(1- \lambda^{a_1} y_1)\cdots (1- \lambda^{a_d} y_d)}
\end{equation}
enumerates the lattice points of the vertex simplicial cone $v+C(H)$, where
\[
v=\left(\begin{array}{c}
    \frac{b}{f} \\
    0 \\
    \vdots \\
    0
    \end{array}\right) \in \QQ^{d+1}, \quad
H = \left(\begin{array}{cccc}
    -a_1 & -a_2 & \cdots & -a_d \\
    f & 0 & \cdots & 0\\
    0 & f & \cdots & 0 \\
    \vdots & \vdots & \ddots & \vdots \\
    0 & 0 & \cdots & f
    \end{array}\right) \in \ZZ^{(d+1)\times d}.
\]
\end{lem}

The cone $C(H)$ is actually the denumerant cone $\D_1(f,a_1,\dots,a_d)$.
Lemmas \ref{lem-cone-1} and \ref{lem-cone-2} reduce the computation of $E_m(\a; f, T)$ to that of
$\sigma_{v+C(H)}(\y)$.
And by Theorem \ref{theo-polytime-barv},
we can write \eqref{e-cone-2} as a short sum of rational functions.
\begin{lem}\label{lem-cone-3}
Following notation in Lemma \ref{lem-cone-2}, we have
\[
  \CT_\lambda \frac{1}{\underline{1- \lambda^f y_0}} \frac{\lambda^{-b}}{(1- \lambda^{a_1} y_1)\cdots (1- \lambda^{a_d} y_d)}
  = \sum_{i \in I} \epsilon_i \frac{\y^{\alpha_{i0}}}{\prod_{j=1}^{d} (1 - \y^{\alpha_{ij}})},
\]
where $I$ is the index set, $\epsilon_i \in \{1, -1\}$, $\alpha_{i1}, \dots, \alpha_{id}$ generate a unimodular cone for each $i$,
and $\alpha_{i0} = \gamma_i + \sum_{j=1}^{d} \lceil k_{ij} \rceil \alpha_{ij}$
by writing $v = \gamma_i + \sum_{j=1}^{d} k_{ij}  \alpha_{ij}$ with $\gamma_i \in \ZZ^{d+1}, k_{ij} \in \QQ$.
\end{lem}
\begin{proof}
To apply Theorem \ref{theo-polytime-barv}, we need to transform $C(H)$ into a full dimensional cone.
To this end, find a unimodular matrix $U \in \ZZ^{(d+1) \times (d+1)}$ such that $U H = \left(\begin{array}{c}
\mathbf{0}^T \\
 H'
\end{array}\right)$, where $H' \in \ZZ^{d \times d}$. We additionally assume the $(1,1)$-entry of $U$ is nonnegative.
See Subsection \ref{ss-exam} on how to compute $U$.

We claim that $U v = \left(\begin{array}{c}
b \\
v'
\end{array}\right)$, where $v'\in \QQ^{d}$. To see this, denote by $U_1$ the first row of $U$.
Then $U_1$ must be primitive and $H^T U_1^T = \mathbf{0}$. It follows that
$U_1 = (f, a_1, \dots, a_d)$. The claim then follows.

Applying Proposition \ref{prop-VCone} and Theorem \ref{theo-polytime-barv} to $C(H')$ gives a unimodular cone decomposition of the form
\[
\sigma_{v' + C(H')}(\y) = \sum_{i\in I} \epsilon_i \sigma_{v' + C(B_i)}(\y)
= \sum_{i \in I} \epsilon_i \frac{\y^{\sum_{j=1}^d \lceil k_{ij} \rceil \beta_{ij}}}{\prod_{j=1}^{d} (1 - \y^{\beta_{ij}})},
\]
where $(k_{i1}, \dots, k_{id})^T = B_i^{-1} v' \in \QQ^d$.
To each $B_i$, let $A_i=U^{-1} \left(\begin{array}{c} \mathbf{0}^T \\ B_i\end{array}\right) $. Then we obtain the unimodular decomposition
\[
\sigma_{\hat v + C(H)}(\y) = \sum_{i\in I} \epsilon_i \sigma_{\hat v + C(A_i)}(\y)
= \sum_{i \in I} \epsilon_i \frac{\y^{\sum_{j=1}^d \lceil k_{ij} \rceil \alpha_{ij}}}{\prod_{j=1}^{d} (1 - \y^{\alpha_{ij}})},
\]
where $\hat v = U^{-1} \left(\begin{array}{c} 0 \\ v' \end{array} \right) = U^{-1} \left(\begin{array}{c} \mathbf{0}^T \\ B_i \end{array} \right) (k_{i1}, \dots, k_{id})^T = A_i (k_{i1}, \dots, k_{id})^T $.

Now set $\gamma_i = U^{-1} \left(\begin{array}{c}
b \\
\mathbf{0}
\end{array}\right) \in \ZZ^{d+1}$. Then we have
\[
v = U^{-1} U v = U^{-1} \left(\begin{array}{c}
b \\
v'
\end{array}\right)
= U^{-1} \left(\begin{array}{c} b \\ \mathbf{0} \end{array}\right) + U^{-1} \left(\begin{array}{c} 0 \\ v' \end{array}\right)
= \gamma_i + \hat v.
\]
Hence
\[
\sigma_{v + C(H)}(\y) = \sum_{i\in I} \epsilon_i \sigma_{v + C(A_i)}(\y)
= \sum_{i\in I} \epsilon_i \y^{\gamma_i} \cdot \sigma_{\hat v + C(A_i)}(\y)
= \sum_{i\in I} \epsilon_i \frac{\y^{\gamma_i + \sum_{j=1}^d \lceil k_{ij} \rceil \alpha_{ij}}}{\prod_{j=1}^{d} (1 - \y^{\alpha_{ij}})}.
\]
This completes the proof.
\end{proof}

\section{The computation of $E_m(\a; f, T)$}\label{s-computation}
In this section, we use cone decompositions to compute $E_m(\a;f,T)$. By \eqref{e-Emf-1}, we have
\begin{equation}\label{e-Emf-2}
E_m(\a; f, T) =\frac{ (-1)^{m+1}}{m!} [s^{-1-m}] \F(\a, f, T; s),
\end{equation}
where
\[
\F(\a,f,T;s)=\sum_{\zeta:\zeta^f=1} \frac{\zeta^{-T}}{\prod_{i=1}^{N+1} (1-\zeta^{a_i} \e^{a_is})}.
\]

Now we split $\F$ as follows.
\[
\F(\a,f,T;s) = \B(\a, f; s) \cdot \S(\a, f, T; s),
\]
where
\[
\B(\a, f; s) = \frac{1}{\prod_{i : f \mid a_i} (1-\e^{a_i s})}, \qquad
\S(\a, f, T; s) =\sum_{\zeta:\zeta^f=1} \frac{\zeta^{-T}}{\prod_{i: f \nmid a_i} (1-\zeta^{a_i} \e^{a_is})}.
\]
In the $f=1$ case, we treat $\S(\a, 1, T; s)=1$.

Suppose $\{i : f \nmid a_i\} = \{i_1,i_2, \dots, i_d\}$.
Then we must have $\gcd(f, a_{i_1}, \dots, a_{i_d}) = 1$, for otherwise $\gcd(\a) \neq 1$.
By Lemma~\ref{lem-cone-1}, we can write
\[
\S(\a, f, T; s) = f \CT_\lambda \frac{1}{\underline{1-\lambda^f}} \frac{\lambda^{-T}}{\prod_{\ell = 1}^d (1-\lambda^{a_{i_\ell}} y_\ell)} \Big|_{y_\ell=\e^{a_{i_\ell} s}}
= f \S'(\a, f, T; s) \Big|_{y_0=1, y_\ell=\e^{a_{i_\ell} s}},
\]
where
\[
\S'(\a, f, T; s) = \CT_\lambda \frac{1}{\underline{1-\lambda^f y_0}} \frac{\lambda^{-T}}{\prod_{\ell = 1}^d (1-\lambda^{a_{i_\ell}} y_\ell)}
\]
has a cone interpretation by Lemma~\ref{lem-cone-2}. Thus we may apply Lemma \ref{lem-cone-3} to obtain
\[
\S'(\a, f, T; s) = \sum_{i \in I} g_i(\y) = \sum_{i \in I} \epsilon_i \frac{\y^{\alpha_{i0}}}{\prod_{j=1}^{d} (1 - \y^{\alpha_{ij}})},
\]
where $I$ is the index set, $\epsilon_i \in \{1, -1\}$, $\alpha_{i1}, \dots, \alpha_{id}$ generate a unimodular cone for each $i$, and $\alpha_{i0}$ is determined by $\alpha_{i1}, \dots, \alpha_{id}$.
In detail, suppose
\[
\left(\begin{array}{c}
    \frac{1}{f} \\
    0 \\
    \vdots \\
    0
    \end{array}\right)
    = \gamma_i + \sum_{j=1}^{d} k_{ij} \alpha_{ij},
\]
where $\gamma_i \in \ZZ^{d+1}$, $k_{ij} \in \QQ$.
Then
\[
\alpha_{i0} = T \gamma_i + \sum_{j=1}^{d} \lceil k_{ij} T \rceil \alpha_{ij}
= T \gamma_i + \sum_{j=1}^{d} k_{ij} T \alpha_{ij} + \sum_{j=1}^{d} \{ -k_{ij} T \} \alpha_{ij}
= \left(\begin{array}{c}
    \frac{T}{f} \\
    0 \\
    \vdots \\
    0
    \end{array}\right) + \sum_{j=1}^{d} \{ - k_{ij} T \} \alpha_{ij}.
\]
Since we will set $y_0 = 1$, we may just take
\[
\alpha_{i0} = \sum_{j=1}^{d} \{ - k_{ij} T \} \alpha_{ij}.
\]
And then we have
\[
\S(\a, f, T; s) = f \sum_{i \in I} g_i(\y) \Big|_{y_0=1, y_\ell=\e^{a_{i_\ell} s}}.
\]

However, direct substitution $y_0=1$, $y_\ell=\e^{a_{i_\ell} s}$ might result in $1-\y^{\alpha_{ij}} \to 0$ in the denominator.
To avoid such situations, we add slack variables by setting $y_0=z_0$, $y_\ell=z_\ell q^{a_{i_\ell}}$ in $\S(\a, f, T; s)$ and take the limit at $z_\ell=1$ for all $\ell$,
and finally set $q=\e^s$.

This type of limit has been addressed in \cite{2015-Xin-CTEuclid}, and discussed further in \cite{2023-Xin-GTodd}. The idea is to first choose an integral vector $(c_0,\dots, c_d)$ and make the substitution $z_\ell = \kappa^{c_\ell}$ such that there is no zero in the denominator, i.e., none of $\y^{\alpha_{ij}}$ will become $1$.
This reduces the number of slack variables to $1$. Now we need to compute the limit at $\kappa=1$.

By letting $\kappa = \e^x$ we can compute separately the constant term of $\hat g_i := g_i |_{y_0=\e^{c_0 x}, y_\ell=\e^{c_\ell x} q^{a_{i_\ell}}} $ in $x$.
Each constant term is a rational function in $q$, denoted $\bar{g}_i(q)$.
Hence
\[
\S(\a, f, T; s) = f \sum_{i \in I} \bar{g}_i(\e^s).
\]
For each $i$, $\hat g_i$ has the structure:
\begin{equation}\label{e-gistru}
\hat g_i = \frac{q^{m_0} \e^{b_0 x}}{\prod_{b \in B_0} (1-\e^{bx})} \prod_{j=1}^{r} \frac{1}{\prod_{b \in B_j} (1-q^{m_j} \e^{bx})} \prod_{j=r+1}^{r+r'} \frac{1}{(1-q^{m_j})}
\end{equation}
where $m_0$ and $b_0$ might be symbolic, $m_j \in \QQ$ for $1\le j \le r+r'$, and $B_j \subset \QQ \setminus \{0\}$ are finite multi-sets for $0 \le j \le r$.

Xin et al. provided an algorithm, called Algorithm \texttt{CTGTodd}, to compute $\CT_x \hat g_i \pmod p$ for a suitable prime $p$
with a good complexity result. The algorithm also works over $\QQ$, but may involve the large integer problem, making it hard to do complexity analysis.
Applying Algorithm \texttt{CTGTodd} to \eqref{e-gistru} gives the following corollary.
See \cite[Section 4]{2023-Xin-GTodd} for details.
\begin{cor}\label{cor-CTGTodd}
Following all above notations, $\bar{g}_i(q)$ can be computed efficiently as a sum of at most $\binom{|B_0| + r + 1}{r + 1}$ simple rational functions in the following form:
\[
\bar{g}_i(q) = \CT_{x} \hat g_i = A \sum_{n=0}^{|B_0|} \frac{b_0^n}{n!} \sum_{l_1 + \cdots + l_r \le |B_0|-n \atop l_1,\dots,l_r \ge 0} C_{l_1,\dots, l_r}
\prod_{j=1}^{r} \Big(\frac{-1}{1 - q^{-m_j}}\Big)^{l_j},
\]
where
\[
A = \frac{(-1)^{|B_0|} q^{m_0}}{(\prod_{b \in B_0} b) \cdot (\prod_{j=1}^{r} (1-q^{m_j})^{|B_j|}) \cdot (\prod_{j=r+1}^{r+r'} (1 - q^{m_j}))}, \quad
C_{l_1,\dots, l_r} \in \QQ.
\]
Note that each term has a monomial numerator and denominator a product of at most $r'+\sum_{j=0}^{r} |B_j|$ binomials.
\end{cor}

By expanding $\bar{g}_i(q) $ as a sum according to Corollary \ref{cor-CTGTodd} for each $i$,
set $q=\e^s$, and then multiply each term by $\B(\a, f; s)$. We obtain
\begin{equation}\label{e-FeqfsumGi}
\F(\a, f, T; s) = f \sum_{i \in I'} G_i(s).
\end{equation}
Now each $G_i(s)$ has the structure:
\begin{equation}\label{e-Gistru}
G_i(s) = \frac{c \cdot \e^{b_0 s}}{\prod_{j=1}^{u} (1 - \e^{b_j s})},
\end{equation}
where $c$ and $b_0$ might be symbolic, $b_j \in \QQ \setminus \{0\}$ for $1\le j \le u$.

Therefore \eqref{e-Emf-2} becomes
\begin{equation}\label{e-Emf-Gi}
E_m(\a; f, T) = \frac{(-1)^{m+1} f}{m!} \sum_{i \in I'} [s^{-1-m}] G_i(s).
\end{equation}
It is trivial that $[s^{-1-m}] G_i(s) = 0$ when $u < m+1$. Indeed, we have
\[
[s^{-1-m}] G_i(s) = [s^{-1-m+u}] \frac{c}{\prod_{j=1}^{u} b_j} \e^{b_0 s} \prod_{j=1}^{u}  \frac{b_j s}{1 - \e^{b_j s}}
= [s^{-1-m+u}] \frac{c}{\prod_{j=1}^{u} b_j} \e^{b'_0 s} \hat{G}_i(s),
\]
where $b'_0 = b_0-\frac{\sum_{j=1}^{u} b_j}{2}$, and
\begin{equation}\label{e-hatGi}
\hat{G}_i(s) = \e^{\frac{\sum_{j=1}^{u} b_j}{2} s} \prod_{j=1}^{u}  \frac{b_j s}{1 - \e^{b_j s}}
\end{equation}
is easily checked to be a power series in $s^2$.
Write $\hat{G}_i(s) = \sum_{k \ge 0} M_k s^k$, where $M_k=0$ for odd $k$.
Xin et al. proved \cite[Theorem 19]{2023-Xin-GTodd} that the sequence $(M_0, M_1, \dots, M_n) \pmod p$ can be computed in time $O(n \log (n) + n u )$. Again, the result holds over $\QQ$, but the complexity should be modified.
Then we have the following result.
\begin{theo}\label{theo-Emf}
Following all above notations, we have
\begin{equation}\label{e-Gi}
[s^{-1-m}] G_i(s) = \begin{cases}
                      \displaystyle \frac{c}{\prod_{j=1}^{u} b_j} \sum_{k=0}^{-1-m+u} \frac{{b'_0}^k}{k!} M_{-1-m+u-k}, & \mbox{if } u \ge m+1;\\
                      0, & \mbox{if } u < m+1.
                    \end{cases}
\end{equation}
Then $E_m(\a; f, T)$ can be obtained directly by substituting the above into \eqref{e-Emf-Gi}.
Furthermore, $E_m(\a; T)$ is computable by Proposition \ref{prop-EmeqsummuEmf}.
\end{theo}

As mentioned in Section \ref{s-reduction}, for a fixed $f$,
we can compute $E_m(\a; f, T)$ simultaneously for all $m \in \m(f)$ since there must be the same $\F(\a, f, T; s)$, hence the same $\{G_i(s) : i \in I'\}$.
This will be completed by the following algorithm.

\begin{algorithm}[H]
\DontPrintSemicolon
\KwInput{$\a$, $f$, $\m(f)$.}
\KwOutput{$E_m(\a; f, T)$ for all $m \in \m(f)$.}

Compute $\F(\a, f, T; s) = f \sum_{i \in I'} G_i(s)$ as \eqref{e-FeqfsumGi}.

For each $G_i(s)$ as in \eqref{e-Gistru}, compute
\[
\hat{G}_i(s) \pmod {\langle s^{u-\min \m(f)} \rangle} =
\sum_{k=0}^{u-\min \m(f)-1} M_k s^k,
\]
where $\hat{G}_i(s)$ is as in \eqref{e-hatGi}.
Then compute $[s^{-1-m}] G_i(s)$ by \eqref{e-Gi} for all $m \in \m(f)$.

Compute $E_m(\a; f, T)$ for all $m \in \m(f)$ by \eqref{e-Emf-Gi}.
\caption{}\label{alg-Emf}
\end{algorithm}

\begin{exam}
Let $\a = (2,3,3,6)$, $\m = \{0,1,2,3\}$.
By Examples \ref{ex-fset} and \ref{ex-Morb},
\begin{align*}
E_0(\a; T) &= E_0(\a; 6, T), &
E_1(\a; T) &= -E_1(\a; 1, T) + E_1(\a; 2, T) + E_1(\a; 3, T), \\
E_2(\a; T) &= E_2(\a; 3, T), &
E_3(\a; T) &= E_3(\a; 1, T).
\end{align*}
Thus to obtain an explicit formula of $E(\a; t)$ as shown in Example \ref{exam-2336}, it suffices to compute $E_m(\a; f, T)$ for $f=1,2,3,6$ and $m\in \m(f)$.

We only give the details for $f=3$. The other cases are similar.
We have $\m(3) = (1, 2)$, and $\F(\a, 3, T; s) = G_1(s) = \displaystyle \frac{\e^{6 \{\frac{2T}{3}\} s}}{(1-\e^{3s})^2 (1-\e^{6s})^2}$.
We first compute
\[
\hat{G}_1(s) \pmod {\langle s^3 \rangle} =
\e^{9s} \frac{324 s^4}{(1-\e^{3s})^2 (1-\e^{6s})^2} \pmod {\langle s^3 \rangle} = 1 - \frac{15}{4} s^2.
\]
Then by Theorem \ref{theo-Emf},
\begin{align*}
E_1(\a; 3, T) &= \frac{(-1)^2 \cdot 3}{1!} \cdot \frac{1}{324} \cdot \Big(-\frac{15}{4} + \frac{(6 \{\frac{2T}{3}\} - 9)^2}{2!}\Big) = \frac{49}{144} - \frac{\{\frac{2T}{3}\}}{2} + \frac{\{\frac{2T}{3}\}^2}{6}, \\
E_2(\a; 3, T) &= \frac{(-1)^3 \cdot 3}{2!} \cdot \frac{1}{324} \cdot \Big(6 \Big\{\frac{2T}{3}\Big\} - 9\Big) = \frac{1}{24} - \frac{\{\frac{2T}{3}\}}{36}.
\end{align*}
\end{exam}

\section{Summary of the algorithm and computer experiments}\label{s-experiments}
We first summarize the main steps to give the following algorithm. The algorithm is implemented as the \texttt{Maple} package \texttt{CT-Knapsack}.

\begin{algorithm}[H]
\DontPrintSemicolon
\KwInput{A positive integer sequence $\a$ with $\gcd(\a)=1$. \\
        \hspace{3.2em} A (symbolic) nonnegative integer $T$. \\
        \hspace{3.2em} The subscript set $\m$ of $E_m(\a; T)$ that we want to compute.}
\KwOutput{$E_m(\a; T)$ for all $m \in \m$.}

Apply Algorithm \texttt{fset} to $\a$ and obtain $\f_m$ for each $m \in \m$.

Apply Algorithm \texttt{M\"obius} to each $\f_m$ to obtain $\mu_m(f)$ for all $f\in \f_m$.

For each $f \in \cup_{m \in \m} \f_m$ satisfying $\m(f) \neq \varnothing$, using Algorithm \ref{alg-Emf} to compute $E_m(\a; f, T)$ for all $m\in \m(f)$.

Compute $E_m(\a; T)$ for all $m\in \m$ using Proposition \ref{prop-EmeqsummuEmf}.

\caption{Algorithm \texttt{CT-Knapsack}}\label{alg-CTKnap}
\end{algorithm}

Then we make some computer experiments on $\a$ in Table \ref{table-instance}. Their running time is reported in Table \ref{table-runningtime}, which also includes the running time by Baldoni et al.'s packages \texttt{M-Knapsack} and \texttt{LattE Knapsack} for comparison. In these experiments, we terminate our procedure
if the running time is longer than $20$ minutes.

Table \ref{table-instance} contains ten selected instances in \cite[Table 1]{2015-Baldoni-Denumerant}, nine small random sequences and six large random sequences.
Here ``small" and ``large" mean that $1 \le a_i \le 20$ and $10000 \le a_i \le 100000$ respectively.
Table \ref{table-runningtime} shows the running time (in seconds) for computing all coefficients of $E(\a; t)$ by \texttt{CT-Knapsack}, \texttt{M-Knapsack} and \texttt{LattE Knapsack} on the same personal laptop. From the table, we see that \texttt{CT-Knapsack} has exactly a significant speed advantage over \texttt{M-Knapsack}. This is due to the following reasons.
\begin{enumerate}
  \item In Algorithm \texttt{fset}, we avoid plenty of $\gcd$ computations.

  \item In Algorithm \texttt{M\"obius}, we avoid repeated computations of $\mu_m(f)$.

  \item For each $f$, we only compute $E_m(\a;f,T)$ for $m=\min \m(f)$. The other cases are read off. This avoids
  repeated unimodular cone decompositions and limit computations.

  \item For limit computations, we use the recent results in \cite{2023-Xin-GTodd}.
\end{enumerate}

These advantages can be adapted to give rise a \texttt{C++} implementation of our algorithm. Such an implementation
can be expected to be faster than \texttt{LattE Knapsack}.

Finally, we remark that the three packages in comparison are designed for computing top coefficients of $E(\a; t)$.
They are not suitable for computing the whole quasi-polynomial $E(\a; t)$.
One example is Problem Sel9: it takes \texttt{CT-Knapsack} $0.703$ seconds to compute the top $8$ coefficients, but more than $20$ minutes to compute
the whole quasi-polynomial;
Another example is Problem Sma7: the whole quasi-polynomial is of degree $11$. \texttt{CT-Knapsack} computes the top $10$ coefficients within $7.266$ seconds,
but compute the single top $11$-th and $12$-th coefficients in $119.562$ 
and $553.547$ 
seconds, respectively.

Sills and Zeilberger's \texttt{Maple} package \texttt{PARTITIONS} can compute $p_k(t)$ for $k$ up to $70$ (See \cite{2012-Zeilberger} for details). This corresponds to compute the whole $E(\a; t)$ in the $\a=(1,2,\dots, k)$ case.
However, it already takes $83.4$ seconds for \texttt{CT-Knapsack} to compute $p_{12}(t)$. The computation of
$p_{70}(t)$ is out of reach even for \texttt{LattE Knapsack}.

The reason is that for bottom coefficients,
we may need to decompose some high dimensional cones, while Barvinok's algorithm is only polynomial when the dimension is fixed.
For instance, if we want to obtain $E_0(\a; 71, T)$, where $\a=(1,2,\dots, 71)$, then
\begin{equation}\label{e-dim70}
 \S'(\a, 71, T; s) = \CT_{\lambda} \frac{1}{\underline{1 - \lambda^{71} y_0}} \frac{\lambda^{-T}}{\prod_{i=1}^{70} (1 - \lambda^i y_i)}
\end{equation}
corresponds to a simplicial cone of dimension $70$.

\begin{small}
\begin{center}
\begin{table}[htp]
\caption{Instances.}\label{table-instance}
  \begin{tabular}{ll}
    \toprule
    Problem &  $\a$  \\
    \midrule
    Sel1 & $(8,12,11)$ \\
    Sel2 & $(5,13,2,8,3)$ \\
    Sel3 & $(5,3,1,4,2)$ \\
    Sel4 & $(9,11,14,5,12)$ \\
    Sel5 & $(9,10,17,5,2)$ \\
    Sel6 & $(1,2,3,4,5,6)$ \\
    Sel7 & $(12223,12224,36674,61119,85569)$\\
    Sel8 & $(12137,24269,36405,36407,48545,60683)$ \\
    Sel9 & $(20601,40429,40429,45415,53725,61919,64470,69340,78539,95043)$ \\
    Sel10 & $(5,10,10,2,8,20,15,2,9,9,7,4,12,13,19)$\\
    \midrule
    Sma1 & $(11, 9, 5, 3, 14, 10)$ \\
    Sma2 & $(2, 19, 20, 19, 4, 11, 12)$ \\
    Sma3 & $(18, 10, 5, 2, 4, 18, 19, 5)$ \\
    Sma4 & $(12, 18, 4, 2, 20, 6, 7, 16, 11)$ \\
    Sma5 & $(7, 13, 7, 12, 17, 19, 8, 6, 5, 14)$ \\
    Sma6 & $(6, 6, 4, 18, 16, 8, 15, 8, 11, 15, 3)$\\
    Sma7 & $(18, 20, 11, 19, 14, 18, 15, 8, 10, 14, 12, 9)$ \\
    Sma8 & $(16, 6, 18, 11, 13, 17, 9, 20, 13, 12, 5, 6, 18)$ \\
    Sma9 & $(20, 1, 14, 20, 17, 6, 14, 6, 11, 6, 2, 19, 3, 15)$\\
    \midrule
    Lar1 & $(75541, 29386, 12347)$ \\
    Lar2 & $(66958, 75047, 71820, 69631)$ \\
    Lar3 & $(36723, 52533, 37999, 86519, 15860)$ \\
    Lar4 & $(71273, 66058, 97201, 48161, 60355, 10311)$ \\
    Lar5 & $(13913, 16811, 21299, 75411, 57053, 64181, 28990)$ \\
    Lar6 & $(87394, 47494, 43580, 46684, 93526, 50784, 55902, 90475)$ \\
    \bottomrule
  \end{tabular}
\end{table}
\end{center}
\end{small}

\begin{small}
\begin{center}
\begin{table}[h!]
\caption{Running time (in seconds) for computing all coefficients of $E(\a; t)$.}\label{table-runningtime}
  \begin{tabular}{lccc}
    \toprule
    Problem & \texttt{CT-Knapsack} & \texttt{M-Knapsack} & \texttt{LattE Knapsack}  \\
    \midrule
    Sel1 & 0 & 0.141 & 0  \\
    Sel2 & 0.328 & 1.922 & 0 \\
    Sel3 & 0.172 & 1.422 & 0 \\
    Sel4 & 0.203 & 4.906 & 0.02 \\
    Sel5 & 0.187 & 5.000 & 0.01 \\
    Sel6 & 0.172 & 11.094 & 0.02 \\
    Sel7 & 0.844 & 14.125 & 0.04\\
    Sel8 & 5.218 & 293.484 & 0.19  \\
    Sel9 & $> 20$ min & $> 20$ min &  87.46\\
    Sel10 & $> 20$ min & $> 20$ min & 66.25 \\
    \midrule
    Sma1 & 0.813 & 39.109 & 0.05  \\
    Sma2 & 3.406 & $> 20$ min & 0.23 \\
    Sma3 & 11.859 & $> 20$ min & 0.30 \\
    Sma4 & 17.828 & $> 20$ min & 0.83 \\
    Sma5 & 76.625 & $> 20$ min & 2.30 \\
    Sma6 & 165.094 & $> 20$ min & 4.01 \\
    Sma7 & 721.031 & $> 20$ min & 13.00 \\
    Sma8 & $> 20$ min & $> 20$ min &  19.22 \\
    Sma9 & $> 20$ min & $> 20$ min & 28.19 \\
    \midrule
    Lar1 & 0.187 & 0.219 & 0  \\
    Lar2 & 2.547 & 15.265 & 0.05 \\
    Lar3 & 26.797 & 925.625 & 0.55 \\
    Lar4 & 390.281 & $> 20$ min & 4.94 \\
    Lar5 & $> 20$ min & $> 20$ min & 66.40 \\
    Lar6 & $> 20$ min & $> 20$ min & 390.89 \\
    \bottomrule
  \end{tabular}
\end{table}
\end{center}
\end{small}

\newpage
\section{Concluding remark}\label{s-cr}
This work is along the line of Baldoni et al.'s polynomial algorithm on top coefficients of Sylvester's denumerant.
We establish an algebraic combinatorial approach, avoid plenty of repeated computations, and develop the package \texttt{CT-Knapsack}
to implement our ideas.
We are planning to give a \texttt{C++} implementation in the near future.
The implementation will not use Barvinok's idea, but use its parallel development in an upcoming paper \cite{2023-Xin-knapsackcone}.
The new algorithm, called \texttt{DecDenu}, is specially designed for denumerant cones. Though it is not proved to be polynomial, its practical running time is much faster than Algorithm \ref{alg-barv} in average.
We hope there will be a decomposition result similar to Theorem \ref{theo-polytime-barv}, and hence giving rise a \texttt{C++} implementation that is much faster than \texttt{LattE Knapsack}.

Another direction is to study the whole coefficients when the entries of $\a$ are not large.
In our framework, the major obstacle is the constant term as in \eqref{e-dim70}.
The corresponding $A_{i_0}(\lambda)$ in \eqref{e-siglecontri} has been addressed in \cite{2004-Xin-Ell} and can be computed by the command \texttt{E\_frac\_single} in the \texttt{Maple} package \texttt{Ell}. The result
is a rational function with a simple denominator, but the numerator has as many as $71^{70}$ terms.
However, what we need is only the specialization at $y_i=q^i$. Then the $71^{70}$ terms collapse to a two
variable polynomial of a reasonable degree. Of course, this basic idea needs further work to be realized.

\medskip
\noindent \textbf{Acknowledgments:}
This work was supported by the National Natural Science Foundation of China (No. 12071311).


\begin{thebibliography}{10}
\bibitem{2004-HardKnapsack}
K. Aardal and A. K. Lenstra, \emph{Hard equality constrained integer knapsacks}, Mathematics of Operations Research {\bf 29} (2004): 724--738.

\bibitem{2002-Agnarsson}
G. Agnarsson, \emph{On the Sylvester denumerants for general restricted partitions}, Proceedings of the Thirtythird Southeastern International Conference on Combinatorics, Graph Theory and Computing (Boca Raton, FL, 2002) {\bf 154} (2002): 49--60.

\bibitem{2018-Aguilo}
F. Aguil\'{o}--Gost and D. Llena, \emph{Computing denumerants in numerical $3$-semigroups}, Quaestiones Mathematicae {\bf 41} (2018): 1083--1116.

\bibitem{2012-Baldoni-LattETop}
V. Baldoni, N. Berline, J. A. De Loera, M. K\"oppe, and M. Vergne, \emph{Computation of the highest coeffcients
of weighted Ehrhart quasi-polynomials of rational polyhedra}, Foundations of Computational Mathematics {\bf 12} (2012): 435--469.

\bibitem{2014-Baldoni-LattE}
V. Baldoni, N. Berline, J. De Loera, B. Dutra, M. K\"oppe, S. Moreinis, G. Pinto, M. Vergne, and J. Wu,
\emph{A users guide for LattE integrale v1.7.2}, Available from URL http://www.math.ucdavis.edu/~latte/, 2014.

\bibitem{2015-Baldoni-Denumerant}
V. Baldoni, N. Berline, J. A. De Loera, B. E. Dutra, M. K\"oppe, and M. Vergne, \emph{Coefficients of Sylvester's denumerant}, Integers {\bf 15} (2015): A11.

\bibitem{1994-Barvinok}
A. I. Barvinok, \emph{A polynomial time algorithm for counting integral points in polyhedra when the dimension is fixed}, Mathematics of Operations Research {\bf 19} (1994): 769--779.

\bibitem{1999-Barvinok}
A. I. Barvinok and J. E. Pommersheim, \emph{An algorithmic theory of lattice points in polyhedra}, New perspectives in algebraic combinatorics {\bf 38} (1999): 91--147.

\bibitem{2006-Barvinok-simplex}
A. I. Barvinok, \emph{Computing the Ehrhart quasi-polynomial of a rational simplex}, Math. Comp. {\bf 75} (2006): 1449--1466.

\bibitem{book-Beck}
M. Beck and S. Robins, \emph{Compting the Continuous Discretely}, Springer, Berlin, 2007.

\bibitem{1943-Bell}
E. T. Bell, \emph{Interpolated denumerants and Lambert series}, American Journal of Mathematics {\bf 65} (1943): 382--386.

\bibitem{1988-Brion-lowerdim}
M. Brion, \emph{Points entiers dans les polyèdres convexes}, Ann. Sci. \'{E}cole Norm. Sup. {\bf 21} (1988): 653--663.

\bibitem{2004-LattE}
J. A. De Loera, R. Hemmecke, J. Tauzer, and R. Yoshida, \emph{Effective lattice point counting in rational convex polytopes}, Journal of Symbolic Computation {\bf 38} (2004): 1273--1302.

\bibitem{1997-Baralg}
M. Dyer and R. Kannan, \emph{On Barvinok's algorithm for counting lattice points in fixed dimension}, Math of Operations Research {\bf 22} (1997): 545--549.

\bibitem{1982-LLL}
A. K. Lenstra, H. W. Lenstra, and L. Lov\'asz, \emph{Factoring polynomials with rational coefficients}, Mathematische Annalen {\bf 261} (1982): 515--534.

\bibitem{1995-Lisonek}
P. Lison\v{e}k, \emph{Denumerants and their approximations}, J. Combin. Math. Combin. Comput. {\bf 18} (1995): 225--232.

\bibitem{book-Frobenius}
J. L. Ramírez Alfonsín, \emph{The Diophantine Frobenius Problem}, Oxford Lecture Series in Mathematics and Its Applications, vol. 30, Oxford University Press, 2005.

\bibitem{book-NumericalSemigroups}
J. C. Rosales and P. A. Garc\'{\i}a-S\'anchez, \emph{Numerical Semigroups}, Developments in Mathematics, vol. 20, Springer, New York, 2009.

\bibitem{book-Minkowski}
A. Schrijver, \emph{Theory of Linear and Integer Programming}, Wiley-Interscience, Chichester, 1986.

\bibitem{2012-Zeilberger}
A. Sills and D. Zeilberger, \emph{Formul{\ae} for the number of partitions of $n$ into at most $m$ parts (using the quasi-polynomial ansatz)}, Adv. in Appl. Math. {\bf 48} (2012): 640--645.

\bibitem{book-Stanley-EC1}
R. P. Stanley, \emph{Enumerative Combinatorics (Volume 1)}, Cambridge studies in advanced mathematics, Cambridge, 2011.

\bibitem{2022-Uday}
N. Uday Kiran, \emph{An algebraic approach to $q$-partial fractions and Sylvester denumerants}, The Ramanujan Journal {\bf 59} (2022): 671--712.

\bibitem{2004-Xin-Ell}
G. Xin, \emph{A fast algorithm for MacMahon's partition analysis}, Electron. J. Combin. {\bf 11} (2004): R58 (electronic).

\bibitem{2015-Xin-CTEuclid}
G. Xin, \emph{A Euclid style algorithm for MacMahon's partition analysis}, Journal of Combinatorial Theory, Series A {\bf 131} (2015): 32--60.

\bibitem{2023-Xin-CTCone}
G. Xin and X. Y. Xu, \emph{A polynomial time algorithm for calculating Fourier-Dedekind sums}, arXiv: 2303.01185, 2023.

\bibitem{2023-Xin-GTodd}
G. Xin, Y. R. Zhang, and Z. H. Zhang, \emph{Fast evaluation of generalized Todd polynomials: applications to MacMahon's partition analysis and integer programming}, arXiv: 2304.13323, 2023.

\bibitem{2023-Xin-knapsackcone}
G. Xin, Y. R. Zhang, and Z. H. Zhang, \emph{A combinatorinal decomposition of knapsack cones}, In preparation.
\end{thebibliography}
\end{document}